\documentclass[12pt]{amsart}

\usepackage[margin=1.25in]{geometry} 

\usepackage{xcolor} 
\usepackage{amsthm,,amssymb}
\usepackage{enumerate,enumitem,nicefrac,listings,longtable,caption,makecell}
\usepackage[pdfencoding=auto,psdextra,colorlinks,linkcolor=blue,citecolor=blue]{hyperref}
\usepackage[pagewise]{lineno}\nolinenumbers 

\newtheorem{theorem}{Theorem}[section]
\theoremstyle{plain}
\newtheorem{lemma}[theorem]{Lemma}
\newtheorem{proposition}[theorem]{Proposition}

\theoremstyle{definition}
\newtheorem{remark}[theorem]{Remark}

\newcommand{\q}[1]{\overline{#1}}
\newcommand{\mc}[1]{\mathcal{#1}}
\newcommand{\mbf}[1]{\mathbf{#1}}
\newcommand{\mr}[1]{\mathrm{#1}}

\newcommand{\nrm}{\trianglelefteq}
\newcommand{\Zen}[1]{\mathbf{Z}(#1)} 
\newcommand{\Cen}[2]{\mathbf{C}_{#1}(#2)} 
\newcommand{\Norm}[2]{\mathbf{N}_{#1}(#2)} 

\renewcommand{\sl}[2]{\mathrm{SL}_{#1}(#2)} 
\newcommand{\pgl}[2]{\mathrm{PGL}_{#1}(#2)} 
\newcommand{\psl}[2]{\mathrm{PSL}_{#1}(#2)} 
\newcommand{\su}[2]{\mathrm{SU}_{#1}(#2)}   
\newcommand{\psu}[2]{\mathrm{PSU}_{#1}(#2)} 
\renewcommand{\sp}[2]{\mathrm{Sp}_{#1}(#2)} 
\newcommand{\psp}[2]{\mathrm{PSp}_{#1}(#2)} 

\definecolor{codegray}{rgb}{0.5,0.5,0.5}
\definecolor{codepurple}{rgb}{0.58,0,0.82}
\definecolor{backcolour}{rgb}{0.95,0.95,0.92}

\lstdefinestyle{mystyle}{
    backgroundcolor=\color{backcolour},   
    commentstyle=\color{black},
    keywordstyle=\color{black},
    numberstyle=\tiny\color{codegray},
    stringstyle=\color{codepurple},
    basicstyle=\ttfamily\footnotesize,
    breakatwhitespace=false,         
    breaklines=true,                 
    captionpos=b,                    
    keepspaces=true,                 
    numbers=left,                    
    numbersep=5pt,                  
    showspaces=false,                
    showstringspaces=false,
    showtabs=false,                  
    tabsize=2
}

\title[Many $p$-regular conjugacy classes]{Finite groups with many $p$-regular conjugacy classes}
\author{Christopher A. Schroeder}
\date{\today}

\begin{document}

\maketitle

\vspace{-7pt}
\begin{quote}
\footnotesize
\textsc{Abstract.} Let $G$ be a finite group and let $p$ be a prime. In this paper, we study the structure of finite groups with a large number of $p$-regular conjugacy classes or, equivalently, a large number of irreducible $p$-modular representations. We prove sharp lower bounds for this number in terms of $p$ and the $p'$-part of the order of $G$ which ensure that $G$ is $p$-solvable. A bound for the $p$-length is obtained which is sharp for odd primes $p$. We also prove a new best possible criterion for the existence of a normal Sylow $p$-subgroup in terms of these quantities. 
\end{quote}
\vspace{15pt}

\section{Introduction}\label{sec:intro}
Let $G$ be a finite group. It was first pointed out by Gustafson \cite{G73} that the invariant $d(G) = k(G)/|G|$ is equal to the probability that two randomly chosen elements of $G$ commute, where $k(G)$ is the number of conjugacy classes of $G$. Many studies have shown that this invariant, often called the commuting probability or commutativity degree, encodes important structural information about the group; see \cite{E15,G06,G73,L95} and the references therein. For example, Gustafson proved in \cite{G73} that if $d(G)>5/8$, then $G$ is abelian. 

There is currently much interest in local versions of $d(G)$. To explain the idea, we need some notation. Let $\pi$ be a collection of primes. A positive integer $n$ is called a $\pi$-number if $\pi$ contains all the primes dividing $n$, and we define $n_\pi$ to be the largest $\pi$-number dividing $n$. An element of $G$ is called a $\pi$-element if its order is a $\pi$-number, and we write $k_\pi(G)$ for the number of conjugacy classes of $\pi$-elements of $G$. The $\pi$-local invariant $d_\pi(G) = k_\pi(G)/|G|_\pi$ captures information about the $\pi$-structure of the group. Indeed, Hung and Mar{\'o}ti proved in \cite{M14} that if $d_{\pi}(G)>5/8$, then $G$ contains an abelian Hall $\pi$-subgroup. Other results in this direction use the primes in $\pi$ to bound $d_{\pi}(G)$. Tong-Viet proved in \cite{TV20} that if $p$ is the smallest prime in a collection $\pi$ of odd primes and $d_\pi(G) > (p+1)/(2p)$, then $G$ has an abelian Hall $\pi$-subgroup and a normal $\pi$-complement. Very recently, Hung, Mar{\'o}ti and Mart{\'i}nez proved in \cite{M23} that if $d_\pi(G)>(p^2+p-1)/p^3$, then $G$ has an abelian Hall $\pi$-subgroup, where $p$ is the smallest prime in an arbitrary collection $\pi$ of primes. In another direction, Burness, Guralnick, Moret{\'o} and Navarro investigated a local probabilistic invariant, showing in \cite{B22} that the probability $\mr{Pr}_p(G)$ that two random $p$-elements of $G$ commute satisfies $\mr{Pr}_p(G)>(p^2+p-1)/p^3$ if and only if $G$ has a normal and abelian Sylow $p$-subgroup, where $p$ is a prime.

As these results indicate, investigations up to this point have been limited to controlling the $\pi$-structure of $G$ using $d_\pi(G)$. However, there are natural and important examples where one would like to control the $\pi'$-structure as well. Let $p$ be a prime and consider $\pi=p'$, the set of all primes not equal to $p$. Then $k_{p'}(G)$ is the number of conjugacy classes of elements whose order is not divisible by $p$, the so-called $p$-regular elements. It is also the number of inequivalent irreducible $p$-modular representations of $G$. Viewed in this light, the invariant $d_{p'}(G)=k_{p'}(G)/|G|_{p'}$ holds special interest because it reflects the $p$-modular representation theory of $G$. For instance, Sangroniz proved in \cite{S08} that if $G$ is $p$-solvable, then all irreducible $p$-Brauer characters of $G$ restrict irreducibly to a proper subgroup $H$ if and only if $d_{p'}(H)=d_{p'}(G)$. Recall that a finite group is $p$-solvable if it has a normal series in which each factor is either a $p$-group or a $p'$-group, and its $p$-length is the minimum possible number of factors that are $p$-groups among such normal series. G. Navarro has kindly provided us with examples, described at the end of Section \ref{sec:normalp}, showing that Sangroniz' result cannot be generalized to arbitrary finite groups.

In this paper, we consider the invariant $d_{p'}(G)$ and prove best possible bounds for controlling the $p$-structure of $G$. The key result is a lower bound ensuring that $G$ is $p$-solvable. We will consider $p$ odd and $p=2$ separately because the behavior is different in these two cases.

Here is our structure theorem for odd primes.

\begin{theorem}\label{thm1:odd} 
Let $G$ be a finite group and let $p$ be an odd prime.
\begin{enumerate}[nolistsep,label=\textup{(\arabic*)}]
    \item Let $d_{p'}(G)>1/(p-1)$. Then $G$ is $p$-solvable; moreover, $G$ has $p$-length at most 2 and the number of nonabelian simple $p'$-factors in a composition series for $G$ is strictly less than $\ln(p-1)/\ln(12)$. In particular, if $p \le 13$, then $G$ is solvable.
    \item Assume $G$ is not $p$-solvable. Let $H=G/\mbf{O}_p(G)$ and $Z=\Zen{H}$. Then $d_{p'}(G)=1/(p-1)$ if and only if $p>3$, $H/Z$ is isomorphic to $\psl{2}{p}$ or $\pgl{2}{p}$, and $H' \cong \psl{2}{p}$.
\end{enumerate}
\end{theorem}

For $p=3$, most of Theorem \ref{thm1:odd}(1) follows by setting $\pi=3'$ in \cite[Theorem C]{TV20}. For $p>3$, Theorem \ref{thm1:odd}(1) is new and part (2) shows the $p$-solvable criterion is best possible. The bound for the $p$-length in Theorem \ref{thm1:odd} is best possible. This will be proved after Theorem \ref{thm:normalp}. Note that in the situation of Theorem \ref{thm1:odd}(1), $G$ need not be solvable for $p>13$, as $A_5$ shows. It follows immediately from \cite[Theorem 1.1]{M23} that $G$ has an abelian Hall $\pi$--subgroup in part (1), where $\pi$ is the collection of primes greater than $p$. 

Next is a structure theorem for $p=2$. Recall that a $2$-solvable group is solvable by the Feit--Thompson Odd Order Theorem.

\begin{theorem}\label{thm2:even}
Let $G$ be a finite group.
\begin{enumerate}[nolistsep,label=\textup{(\arabic*)},resume]
    \item Let $d_{2'}(G)>4/15$. Then $G$ is solvable; moreover, $G$ has $2$-length at most 4.
    \item Assume $G$ is nonsolvable. Let $H=G/\mbf{O}_2(G)$ and $Z=\Zen{H}$. Then $d_{2'}(G)=4/15$ if and only if $H \cong A_5 \times Z$.
\end{enumerate}
\end{theorem}

Again, it follows immediately from \cite[Theorem 1.1]{M23} that $G$ has an abelian Hall $\pi$--subgroup in part (1), where $\pi$ is the collection of primes greater than $3$. 

Using the $p$-solvability criteria in Theorems \ref{thm1:odd} and \ref{thm2:even}, we may deduce a criterion for the existence of a normal Sylow $p$-subgroup. 

\begin{theorem}\label{thm:normalp}
Let $G$ be a finite group and let $p$ be a prime. If $d_{p'}(G)>2/(p+1)$, then $G$ has a normal Sylow $p$-subgroup.
\end{theorem}

The bound in Theorem \ref{thm:normalp} is best possible. Suppose $p$ is a Mersenne prime $p=2^n-1$ and let $G=C_2^n \rtimes C_p$, where $C_p$ acts transitively on $C_2^n-\{1\}$. Then $k_{p'}(G) = 2$ and $|G|_{p'}=2^n=p+1$, so $d_{p'}(G)=2/(p+1)$ and $G$ does not have a normal Sylow $p$-subgroup. The bound is also sharp for $p=2$. If $G=S_3$, then $d_{2'}(G)= 2/3$ and $G$ does not have a normal Sylow $2$-subgroup.

Now we prove that the bound for the $p$-length in Theorem \ref{thm1:odd} is also best possible. Consider again the Frobenius group $G=C_2^n \rtimes C_p$ where $p=2^n-1$ is a Mersenne prime. Then $G$ acts on the vector space $V \cong \mr{GF}(p)^{2^n}$ by permuting basis vectors according to its action on the $2^n$ cosets of $C_p$ in $G$. The one-dimensional subspace of $V$ spanned by the sum of basis vectors is $G$-invariant, and the complementary subspace $W$ is an irreducible $p$-modular representation of $G$ of degree $2^n-1=p$. Then $H=W \rtimes G \cong C_p^p \rtimes (C_2^n \rtimes C_p)$ has $p$-length 2 and $d_{p'}(H) \ge 2/2^n>1/(p-1)$ when $p>3$. For $p=2$, the symmetric group $S_4$ has $2$-length $2$ and $d_{2'}(G)=2/3 > 4/15$. We have been unable to construct a finite group $G$ of $2$-length greater than $2$ which satisfies $d_{2'}(G)>4/15$.

Finally, if $G$ is $p$-solvable of $p$-length $k$, we may deduce an upper bound for $d_{p'}(G)$ using the normal Sylow $p$-subgroup criterion of Theorem \ref{thm:normalp}.

\begin{theorem}\label{thm4:plength} 
Let $G$ be a finite group and let $p$ be a prime. If $G$ is $p$-solvable of $p$-length $k$, then $d_{p'}(G) \le \big( 2/(p+1) \big)^{k-1}$.
\end{theorem}

Our paper is organized as follows. In Section \ref{sec:dp'} we discuss some important properties of $d_{\pi}(G)$ and prove some necessary lemmas. We prove the $p$-solvability criteria of Theorem \ref{thm1:odd}(1) and Theorem \ref{thm2:even}(1) in Section \ref{sec:psolvableodd}. These proofs use the Classification of Finite Simple Groups. In Section \ref{sec:normalp}, we prove Theorem \ref{thm:normalp}. We  prove Theorem \ref{thm4:plength} as well as the remaining parts of Theorem \ref{thm1:odd} and Theorem \ref{thm2:even} first in Section \ref{sec:normalp} because the proofs require our $p$-solvability criterion and our normal Sylow $p$-subgroup criterion. The proof of Theorem \ref{thm1:odd} involves some straightforward GAP \cite{GAP4} calculations which we include in Section \ref{sec:calculations}.

\section{Properties of \texorpdfstring{$d_\pi(G)$}{dpi(G)}}\label{sec:dp'}
In this section, we discuss the properties of $d_\pi(G)$ that we will need in the rest of the paper.
\begin{lemma} \label{lemma:1}
Let $G$ be a finite group and $\pi$ a collection of primes.
\begin{enumerate}[nolistsep,label=\textup{(\arabic*)}]
    \item We have $d_\pi(G) \le 1$. If $N \nrm G$, then $d_\pi(G) \le d_\pi(G/N) d_\pi(N)$. In particular, $d_\pi(G) \le d_\pi(G/N)$ and $d_\pi(G) \le d_\pi(N)$.
    \item If $\mu \subseteq \pi$, then $d_\pi(G) \le d_\mu(G)$.
\end{enumerate}
\end{lemma}

\begin{proof}
The first claim in part (1) is \cite[Lemma 3.5]{M21}, the second claim is \cite[Lemma 2.3]{F12}. Part (2) is essentially due to G.R. Robinson and can be found in \cite[Proposition 5]{M14}. 
\end{proof}

Given a finite group $G$, define $\pi(G)$ to be the collection of primes dividing $|G|$. In the course of our proofs, we will require that the set $\pi(G)$ is sufficiently rich. The following lemma addresses the cases in which the order of a nonabelian finite simple group is divisible by few distinct primes.

\begin{lemma}\label{lemma:2}
The order of a nonabelian finite simple group is divisible by at least 4 distinct primes except if $G$ is one of the groups in Table \ref{table:4primes}. In each of these cases, $d_{2'}(G) \le 4/15$, with equality if and only if $G \cong A_5$.
\end{lemma}

\begin{proof}
It is proved in \cite{H00} that the groups listed in Table \ref{table:4primes} are the only nonabelian finite simple groups whose order is not divisible by four distinct primes. The values of $d_{2'}(G)$ can be confirmed with the GAP \cite{GAP4} code given in in Section \ref{sec:calculations}. 
\end{proof}

Finally, we recall the notion of a $\pi$-witness as defined in \cite{M21}: Let $\pi$ be a collection of primes. A subgroup $H \le G$ is a $\pi$-witness for $G$ if $H$ is a $\pi$-subgroup and $k_\pi(G) \le k_\pi(H)=k(H)$. That is, the number of $\pi$-classes of $G$ is bounded above by the number of $\pi$-classes of a $\pi$-subgroup. If a $\pi$-witness for $G$ exists, we say that $G$ is $\pi$-bounded. Malle, Navarro and Robinson proved in \cite[Theorem 2]{M21} that all finite quasisimple groups are $p'$-bounded for every prime $p$. The explicit $p'$-witnesses constructed in \cite{M21} will be important ingredients in our proofs.

\small
{\renewcommand{\arraystretch}{1.2}
\begin{table}[t]
    \begin{tabular}{ccc}
    \hline
    $G$ & $\pi(G)$ & $d_{2'}(G)$ \\ \hline
    $A_5 \cong \psl{2}{4} \cong \psl{2}{5}$ & $2,3,5$ & $4/15 \approx 0.2667$\\
    $A_6 \cong \psl{2}{9}$ & $2,3,5$ & $5/45 \approx 0.1111$ \\
    $\mathrm{PSp}_4(3) \cong \mathrm{PSU}_4(2)$ & $2,3,5$ & $8/405 \approx 0.0198$\\ 
    $\psl{2}{7} \cong \psl{3}{2}$ & $2,3,7$ & $4/21 \approx 0.1905$\\
    $\psl{2}{8}$ & $2,3,7$ & $8/63 \approx 0.1270$\\
    $\mathrm{PSU}_3(3)$ & $2,3,7$ & $5/189 \approx 0.0265$\\ 
    $\psl{3}{3}$ & $2,3,13$ & $7/351 \approx 0.0199$\\ 
    $\psl{2}{17}$ & $2,3,17$ & $7/153 \approx 0.0458$\\ \hline
    \end{tabular}
    \caption{The nonabelian finite simple groups whose order is not divisible by four distinct primes.} \label{table:4primes}
\end{table}
\normalsize

\section{\texorpdfstring{$p$}{p}-solvability criteria}\label{sec:psolvableodd}
In this section, we first use the Classification of Finite Simple Groups to show that nonabelian finite simple groups with an odd prime $p$ dividing their order do not satisfy the hypothesis of Theorem \ref{thm1:odd}(1).

\begin{proposition}\label{prop:psolvableodd_simplegroups}
Let $G$ be a nonabelian finite simple group and let $p$ be an odd prime. If $p$ divides $|G|$, then $d_{p'}(G) \le 1/(p-1)$. Furthermore, equality holds if and only if $p>3$ and $G \cong \psl{2}{p}$.
\end{proposition}  

We will prove Proposition \ref{prop:psolvableodd_simplegroups} using the proof strategies in the following remarks.
\begin{remark}\label{remark:1}
If $G$ has a $p'$-witness $H \le G$ with $p|H|<|G|_{p'}$, then we may proceed using the following inequalities:
\begin{linenomath}\begin{align*}
    k_{p'}(G) \le k(H) \le |H| < \frac{p \, |H|}{p-1} < \frac{|G|_{p'}}{p-1}.
\end{align*}\end{linenomath}
Dividing by $|G|_{p'}$ yields $d_{p'}(G)<1/(p-1)$, as desired. The inequality $p|H|<|G|_{p'}$ can be proved by examining order formulas for the nonabelian finite simple groups.
\end{remark} 

The next remark describes a computational strategy that we will apply to the sporadic groups and some small alternating groups.

\begin{remark}\label{remark:2}
In order to prove that a nonabelian finite simple group $G$ satisfies Proposition \ref{prop:psolvableodd_simplegroups}, it suffices to show that $d_2(G) < 1/(\widetilde{p}-1)$, where $\widetilde{p}$ is the largest prime dividing $|G|$. Then for each odd prime $p$ dividing $|G|$, we have 
\begin{linenomath}\begin{align*}
    d_{p'}(G) \le d_2(G) < 1/(\widetilde{p}-1) \le 1/(p-1),
\end{align*}\end{linenomath}
where the first inequality follows from Lemma \ref{lemma:1}(2) since $G$ nonabelian simple means $2 \in p'$. The GAP code for this computational strategy is given in Section \ref{sec:calculations}.
\end{remark}

Note that the arguments contained in Remarks \ref{remark:1} and \ref{remark:2} show that $d_{p'}(G)< 1/(p-1)$. We will easily be able to keep track of those groups $G$ that saturate the bound $d_{p'}(G)=1/(p-1)$ since they will require a different method of proof.

We now consider each family of nonabelian finite simple groups in turn. We begin with the alternating groups.

\begin{proposition} \label{prop:alternating}
If $G$ is a simple alternating group $A_n$ with $n \ge 5$ and $p$ is an odd prime dividing $|G|$, then $d_{p'}(G) \le 1/(p-1)$. Furthermore, equality holds if and only if $p=5$ and $G = A_5 \cong \psl{2}{5}$.
\end{proposition}

\begin{proof}
Let $G = A_n$ with $n \ge 5$ and let $p$ be an odd prime dividing $|G|$. By Theorem \cite[Theorem 1]{M21}, $G$ is $\pi$-bounded for any set of primes $\pi$. In fact, the authors show in the proof of \cite[Corollary 3.11]{M21} that if $2 \in \pi$ and $n \ge 10$, then there exists a $\pi$-witness that is a $2$-subgroup of $G$. So assume first that $n \ge 10$. Since $p$ is odd, we have $2 \in p'$ and we may choose a 2-subgroup $H \le G$ as a $p'$-witness. Then $|H| \le |G|_2$ and it suffices to show that $p |G|_2 < |G|_{p'}$ by Remark \ref{remark:1}. If there exists a prime $q>p$ dividing $|G|$, then $p|G|_2 < q |G|_2 \le |G|_{p'}$. So we may assume that $p$ is the largest prime dividing $|G|=n!/2$. Since $n \ge 10$, this means $p \ge 7$. Then $p-2$ and $p-4$ are odd factors of $|G|$ which are coprime to $p$ and $(p-2)(p-4)>p$. Hence, $p|G|_2 < (p-2)(p-4)|G|_2 \le |G|_{p'}$, as desired. For $5 \le n<10$, we compute that $d_{p'}(A_n) \le 1/(p-1)$, with equality if and only if $p=5=n$. The calculations are given in Section \ref{sec:calculations}.
\end{proof}

Next we consider the simple sporadic groups and the Tits group ${}^2 \mr{F}_4(2)'$.

\begin{proposition}\label{prop:sporadic}
If $G$ is a sporadic simple group or the Tits group ${}^2 \mr{F}_4(2)'$ and $p$ is an odd prime dividing $|G|$, then $d_{p'}(G)<1/(p-1)$.
\end{proposition}

\begin{proof}
The claim follows by computing in GAP. The calculations are in Section \ref{sec:calculations}.
\end{proof}

We now address the finite simple groups of Lie type. The proof is divided into two fundamentally different cases depending on whether the odd prime $p$ is the defining characteristic. Our reference for the finite groups of Lie type is \cite{M11}, especially the order formulas given there in Corollary 24.6, Table 24.1 and Table 24.2.

\small
{\renewcommand{\arraystretch}{1.2}
\begin{table}[ht]
\centerline{
    \begin{tabular}{|c|c|c||c|c|c|} 
\hline 
$\mc{G}^F$  
& $|T|$
& $k$
& $\mc{G}^F$  
& $|T|$  
& $k$ \\ \hline

\makecell{$\sl{n}{q}$\\ $n\ge 3$} & $(q^n-1)/(q-1)$  & $\Phi_2(q)$ &
$\mr{F}_4(q)$  & $\Phi_2(q)^4$ &  $\Phi_6(q)$  \\ \hline

\makecell{$\su{n}{q}$ \\ $n\ge 3$} & $\Phi_2(q)^{n-1}$ &  $\Phi_6(q)$ &
$\mr{E}_6(q)$  & $\Phi_3(q)^3$  &  $\Phi_6(q)$  \\ \hline

\makecell{$\mr{Spin}_{2n+1}(q)$ \\ $n \ge 2$} & $\Phi_2(q)^n$&  $\Phi_4(q)$ &  
${}^2 \mr{E}_6(q)$ & $\Phi_2(q)^6$  &  $\Phi_6(q)$  \\ \hline

\makecell{$\sp{2n}{q}$ \\ $n \ge 2$} & $\Phi_2(q)^n$ &  $\Phi_4(q)$ & 
$\mr{E}_7(q)$  & $\Phi_2(q)^7$ &  $\Phi_6(q)$  \\ \hline

\makecell{$\mr{Spin}_{2n}^+(q)$ \\ $n \ge 4$} & $\Phi_4(q)(q^{n-2}+1)$ & $\Phi_2(q)$ &
$\mr{E}_8(q)$  & $\Phi_2(q)^8$ &  $\Phi_8(q)$    \\ \hline

\makecell{$\mr{Spin}_{2n}^-(q)$ \\ $n \ge 4$} & \makecell{$\Phi_2(q)^n$, $n$ odd \\ $\Phi_2(q)^{n-2}\Phi_4(q)$, $n$ even} & $\Phi_6(q)$ &
\makecell{${}^2 \mr{G}_2(q^2)$ \\ $q^2=3^{2f+1}$ \\ $f \ge 1$} & $\Phi_6(q^2)^+$ & $\Phi_1(q^2)$ \\ \hline 

$\mr{G}_2(q)$ & $\Phi_2(q)^2$  & $\Phi_6(q)$ & 
${}^3 \mr{D}_4(q)$ & $\Phi_3(q)^2$  & $\Phi_6(q)$ \\ \hline
\end{tabular}
}
\caption{Large tori in finite groups of Lie type in characteristic $p$ whose order bounds the number of $p'$-conjugacy classes. The $p'$-numbers $k>p$ divide $[G:H]_{p'}$, where $G=\mc{G}^F/Z$, $H=T/Z$ and $Z=\Zen{\mc{G}^F}$.}\label{table:tori}
\end{table}}
\normalsize

\begin{proposition} \label{prop:liedefining}
If $G$ is a finite simple group of Lie type in odd characteristic $p$, then $d_{p'}(G) \le 1/(p-1)$. Furthermore, equality holds if and only if $p>3$ and $G = \psl{2}{p}$.
\end{proposition}

\begin{proof}
Let $G$ be a finite simple group of Lie type in odd characteristic $p$, and let $q$ be a power of $p$. We have $G =\mc{G}^F/\Zen{\mc{G}^F}$ for some simple algebraic group $\mc{G}$ of simply connected type and $F:\mc{G}\rightarrow \mc{G}$ a Steinberg endomorphism. We do not consider types ${}^2 \mr{B}_2$ or ${}^2 \mr{F}_4$ because they occur only in even characteristic. By \cite[Proposition 4.3]{M21} and its proof, $G$ has a $p'$-witness $H$ of the form $H = T/Z$ for some torus $T \ge Z$ of $\mc{G}^F$ and $Z=\Zen{\mc{G}^F}$. By Remark \ref{remark:1}, it suffices to find a $p'$-number $k > p$ dividing $[G:H]_{p'}=[\mc{G}^F:T]_{p'}$, since then $p |H| < k |H| \le |G|_{p'}$. In Table \ref{table:tori}, we give such an integer $k$ for all types except $G = \psl{2}{q} \cong \psu{2}{q}$. We address this case in the next paragraph. The results in Table \ref{table:tori} are expressed in terms of cyclotomic polynomials. Let $\Phi_n(x)$ be the $n$-th cyclotomic polynomial, so $x^n-1 = \prod_{d \mid n} \Phi_d(x)$. In particular,
\begin{linenomath}\begin{align*}
    \Phi_1(x)=x-1, \quad \Phi_2(x)=x+1, \quad \Phi_3(x)=x^2+x+1, \\
    \Phi_4(x)=x^2+1, \quad \Phi_6(x)=x^2-x+1, \quad \Phi_8(x)= x^4+1.
\end{align*}\end{linenomath}
We also define the so-called semicyclotomic polynomials $\Phi_6^\pm(x)=x \pm \sqrt{3x}+1$. Note that $\Phi_6(x)=\Phi_6^+(x) \Phi_6^-(x)$. In Table \ref{table:tori}, we list the order of the torus $T \le \mc{G}^F$ given in the proof of \cite[Proposition 4.3]{M21} for the reader's convenience. Note that the factor $k$ of $|\mc{G}^F|$ survives upon dividing $|\mc{G}^F|$ by $|T|$. Furthermore, $\Phi_n(q)>q \ge p$ for $n=2,4,6,8$ and $\Phi_1(q^2)> 3$ when $q^2=3^{2f+1}$ with $f \ge 1$. Therefore, the $p'$-number $k$ divides $[G:H]_{p'}$ and exceeds $p$, as desired.

It remains to consider $G = \psl{2}{q}$. In this case, $|G| =q(q+1)(q-1)/2$ and it is well-known that $G$ has $k_{p'}(G)=(q+1)/2$ inequivalent irreducible modular representations in defining characteristic. Therefore, $d_{p'}(G) = 1/(q-1) \le 1/(p-1)$, with equality if and only if $q=p$. If we have equality, then $p>3$ since $\psl{2}{3} \cong A_4$ is solvable.
\end{proof}

Now we turn to the finite simple groups of Lie type in cross characteristic. First we address those cases in which there exists a unipotent $p'$-witness.

\begin{proposition} \label{prop:liecrosswitness}
Let $G$ be a finite simple group of Lie type in characteristic $r$, and suppose $G$ is not one of $\psl{2}{q}$, $\psl{3}{q}$, $\psu{3}{q}$ or $\psu{4}{2} \cong \psp{4}{3}$, where $q$ is a power of the prime $r$. If $p$ is an odd prime dividing $|G|$ and $p\neq r$, then $d_{p'}(G) < 1/(p-1)$.
\end{proposition}

\begin{proof}
Let $G$ be a finite simple group of Lie type in characteristic $r$ that is not one of the exceptions listed in the statement of the proposition. Then $G = \mc{G}^F/\Zen{\mc{G}^F}$ for some simple algebraic group $\mc{G}$ of simply connected type equipped with a Steinberg endomorphism $F:\mc{G} \rightarrow \mc{G}$. Let $p \neq r$ be an odd prime dividing $|G|$. By \cite[Theorem 2.9]{M21}, there exists an $r$-subgroup $H \le G$ such that $k(G) \le k(H)$. Then $k_{p'}(G) \le k(G) \le k(H) \le |H| \le |G|_r$ and it suffices to show that $ p |G|_r < |G|_{p'}$ by Remark \ref{remark:1}. Now, if $\mc{G}$ has rank $s$, then
\begin{linenomath}\begin{align*}
    |G|_{r'} = \frac{1}{|\Zen{\mc{G}^F}|}\prod_{i=1}^s (q^{d_i}-\epsilon_i)
\end{align*}\end{linenomath}
for some positive integers $d_i$, roots of unity $\epsilon_i$, and $q$ a power of $r$. To prove the proposition, it suffices to find a $p'$-number dividing $|G|_{r'}$ that is greater than $p$. 

Choose a labeling for the $d_i$ such that $d_i\le d_j$ if $i \le j$.

(1) Suppose $\epsilon_i = \pm 1$ for all $i=1,\dots, s$ and $q^{d_s}-\epsilon_s=q^{2k}-1$ for some $k>1$. 

By examining the order formulas for the finite simple groups of Lie type, this means $G$ is one of the following groups:
\begin{enumerate}[nolistsep,label=(\roman*)]
    \item $\psl{n}{q}$ or $\psu{n}{q}$ with $n$ even. By our assumptions, $n \ge 4$.
    \item $\psp{2n}{q}$ with $n \ge 2$, $\mr{P}\Omega_{2n+1}(q)$ with $n\ge 3$, $\mr{P}\Omega_{2n}^+(q)$ with $n \ge 4$ or $\mr{P}\Omega_{2n}^-(q)$ with $n\ge 4$.
    \item $\mr{G}_2(q)$, $\mr{F}_4(q)$, $\mr{E}_8(q)$, or one of $\mr{E}_6(q)$, ${}^2 \mr{E}_6(q)$, $\mr{E}_7(q)$ modulo its center.
\end{enumerate}
Since $\epsilon_i = \pm1$, $d_i < d_j$ implies that $q^{d_i}-\epsilon_i \le q^{d_j}-\epsilon_j$. We may assume that $p$ divides the largest factor $q^{d_s}-\epsilon_s=q^{2k}-1$ of $|G|_{r'}$. If not, then this is a $p'$-number greater than $p$ and we are done. Since $q^{2k}-1=(q^k-1)(q^k+1)$, we may assume that $p$ divides $q^k+1$ by the same reasoning. Thus, $q^k-1$ is a $p'$-number. Since $d_1=2$ and $|\Zen{\mc{G}^F}|\le q+1$ for all types, we have $z:=(q^2-1)/|\Zen{\mc{G}^F}| > 1$ unless $\mc{G}^F={}^2 \mr{E}_6(2)$. In this case, take $z=2^6-1$. Then $z(q^k-1)$ divides $|G|_{r'}$. We may assume that $p \nmid z$, since otherwise $p \le z \le q^k-1$ and we are done. Thus, $z(q^k-1)$ is a $p'$-number exceeding $q^k+1 \ge p$. This case is finished.

(2) Suppose $G$ is $\psl{n}{q}$ or $\psu{n}{q}$ with $n \ge 5$ odd. 

Assume first that $G = \psl{n}{q}$ with $n \ge 5$ odd. Again, we may assume that $p$ divides the largest factor $q^n-1$ of $|G|_{r'}$. We may also assume that $p$ does not divide $q^{n-1}-1$, for if $p \mid (q^{n-1}-1)$, we may finish using the argument in (1) since $n-1$ is even. Finally, we may assume that $p \nmid (q^{n-2}-1)$; otherwise, $q^{n-1}-1$ is a $p'$-number larger than $p$ and we are done. Thus, $(q^{n-1}-1)(q^{n-2}-1)$ is a $p'$-number larger than $q^n-1 \ge p$, which is what we wanted to find. 

If $G = \psu{n}{q}$, the argument follows in exactly the same way: We may assume that $p$ divides the largest factor $q^n-(-1)^n=q^n+1$ and reduce to finding a $p'$-number $(q^{n-1}-1)(q^{n-2}+1)$ exceeding $p$.

(3) Suppose $G$ is one of the remaining nonabelian finite simple groups of Lie type; namely, ${}^2 \mr{B}_2(q^2)$, ${}^2 \mr{G}_2(q^2)$, ${}^3 \mr{D}_4(q)$ or ${}^2 \mr{F}_4(q^2)$. 

Consider $G = {}^2 \mr{B}_2(q^2)$. Then $r=2$, $q^2=2^{2f+1}$ with $f \ge 1$, and
\begin{linenomath}\begin{align*}
    |G|_{r'} = (q^2-1)(q^4+1)= (q^2-1)(q^2-\sqrt{2q^2}+1)(q^2+\sqrt{2q^2}+1).
\end{align*}\end{linenomath}
We may assume that $p$ divides the largest factor $q^2+\sqrt{2q^2}+1$. Then $p$ cannot divide $q^2-\sqrt{2q^2}+1$ as well, since otherwise $p$ divides their difference $2\sqrt{2q^2}$, contradicting that $p\neq r$. Now, if $p$ divides $q^2-1$, then $p$ divides 
\begin{linenomath}\begin{align*}
    (q^2+\sqrt{2q^2}+1) + (q^2-1) &= 2q^2+\sqrt{2q^2} = 2^{f+1}(2^{f+1}+1) \quad \text{and} \\
    (q^2+\sqrt{2q^2}+1) - (q^2-1) &= \sqrt{2q^2}+2 = 2(2^{f}+1).
\end{align*}\end{linenomath}
Then since $p$ is odd, $p$ divides $(2^{f+1}+1)-(2^f+1)=2^f$, a contradiction. Thus, $q^2-1$ is a $p'$-number and $(q^2-1)(q^2-\sqrt{2q^2}+1)$ is a $p'$-number exceeding $q^2+\sqrt{2q^2}+1 \ge p$.

Consider $G = {}^2 \mr{G}_2(q^2)$. Then $r=3$, $q^2 = 3^{2f+1}$ with $f \ge 1$, and
\begin{linenomath}\begin{align*}
    |G|_{3'} &= (q^2-1)(q^6+1) \\
    &= (q^2-1) (q^2+1) (q^2-\sqrt{3q^2}+1)(q^2+\sqrt{3q^2}+1).
\end{align*}\end{linenomath}
We may assume that $p$ divides the largest factor $q^2+\sqrt{3q^2}+1$. Then $p$ cannot divide $q^2-\sqrt{3q^2}+1$ as well, since otherwise $p$ divides their difference $2\sqrt{3q^2}$, contradicting that $p$ is odd and $p\neq r$. For the same reason, $q^2+1$ is a $p'$-number. So $(q^2+1)(q^2-\sqrt{3q^2}+1)$ is a $p'$-number exceeding $q^2+\sqrt{3q^2}+1 \ge p$.

Consider $G = {}^3 \mr{D}_4(q)$ with $q$ a power of $r$ and
\begin{linenomath}\begin{align*}
    |G|_{r'} &= (q^2-1)(q^6-1)(q^8+q^4+1) \\
    &= (q-1)^2 (q+1)^2 (q^2-q+1)^2 (q^2+q+1)^2 (q^4-q^2+1).
\end{align*}\end{linenomath}
We may assume that $p$ divides the largest factor $q^4-q^2+1$. If $p$ divides $q^2+q+1$, then $p \nmid (q+1)$; otherwise, we obtain the contradiction $p \mid q^2$. Then $(q+1)^2$ is a $p'$-number greater than $q^2+q+1 \ge p$ and we are done. So we may assume $(q^2+q+1)^2$ is a $p'$-number, and it is greater than $q^4-q^2+1 \ge p$.

Finally, consider $G = {}^2 \mr{F}_4(q)$. Then $r=2$, $q^2=2^{2f+1}$ with $f \ge 1$, and
\begin{linenomath}\begin{align*}
    |G|_{2'} &= (q^2-1)(q^6+1)(q^8-1)(q^{12}+1) \\
    &=(q^2-1)^2 (q^2+1)^2 (q^4-q^2+1)(q^4+1)^2(q^8-q^4+1).
\end{align*}\end{linenomath}
We may assume that $p$ divides the largest factor $q^8-q^4+1$. If $p$ divides $(q^4+1)^2$, then $p$ divides $(q^4+1)^2-(q^8-q^4+1)=3q^4$, which means $p=3$ since $p \neq r$. Since $G$ is simple, $|G|$ is divisible by at least three distinct primes. So there certainly exists a prime greater than 3 dividing $|G|_{2'}$ and we are done in this case. If $p$ does not divide $(q^4+1)^2$, then $(q^4+1)^2$ is a $p'$-number exceeding $q^8-q^4+1 \ge p$.
\end{proof}

The remaining cases to consider are the simple groups of Lie type in cross characteristic without a unipotent $p'$-witness. Since $k_{p'}(G) \le k(G)$, it suffices to show that $p \, k(G) < |G|_{p'}$ by a slight modification of the argument described in Remark \ref{remark:1}. Since we are not assured that our ``witness'' $k(G)$ is a $p'$-number, there are more possibilities to consider in each case.

\begin{proposition} \label{prop:liecrossnowitness}
Let $G$ be a finite simple group isomorphic to one of $\psl{2}{q}$, $\psl{3}{q}$, $\psu{3}{q}$ or $\psu{4}{2}\cong \psp{4}{3}$, where $q$ is a power of the prime $r$. If $p$ is an odd prime dividing $|G|$ and $p \neq r$, then $d_{p'}(G) \le 1/(p-1)$. Furthermore, equality holds if and only if $p=7$ and $G = \psl{3}{2} \cong \psl{2}{7}$.
\end{proposition}

\begin{proof}
(1) If $G = \psu{4}{2}$, we may use the argument in Remark \ref{remark:2} and compute that $d_2(G) = 5/64 < 1/4 = 1/(\widetilde{p}-1)$, where $\widetilde{p}$ is the largest prime divisor of $|G|$.

\vspace{0.5em}

(2) Now we consider $G$ isomorphic to $\psl{3}{q}$ or $\psu{3}{q}$. If $G = \psl{3}{2} \cong \psl{2}{7}$, then $\pi(G)=\{2, 3, 7\}$. Calculating in GAP, we have $d_{3'}(G) = 5/56 < 1/(3-1)$ and $d_{7'}(G)=4/24 = 1/(7-1)$, as required. The group $\psu{3}{2}$ is solvable. Hence, we may assume in what follows that $q \ge 3$. Note that since $p \neq r$, the integer $q$ is a $p'$-number. The orders and class numbers of $\psl{3}{q}$ and $\psu{3}{q}$ are given in \cite{S73}.

(2.1) Suppose $G = \psl{3}{q}$ with $(3,q-1)=1$. We have 
\begin{linenomath}\begin{align*}
    |G| =q^3 (q-1)^2 (q+1) (q^2+q+1) \quad \text{and} \quad k_{p'}(G) < k(G) = q^2+q.
\end{align*}\end{linenomath}
So it suffices to show that $p(q^2+q) < |G|_{p'}$. First suppose that $p \mid (q+1)$. Then $p \nmid (q^2+q+1)$ since otherwise $p$ divides $q^2$. Furthermore, $p \le q+1$, so $p<q^2$. Then $p (q^2+q) < p(q^2+q+1) < q^2(q^2+q+1) < |G|_{p'}$, as desired. So we may assume that $p \nmid q+1$. This means $q(q+1)$ is a $p'$-number. We may also assume that $p$ divides the largest factor $q^2+q+1$ of $|G|$. If $p \nmid q-1$, then $q^2(q-1)^2$ is a $p'$-number greater than $q^2+q+1 \ge p$ when $q \ge 3$. So $pq(q+1) < q^2(q-1)^2q(q+1) \le |G|_{p'}$, as wanted. On the other hand, if $p \mid (q-1)$, then $p$ divides $(q^2+q+1)-(q-1)^2 = 3q$, so $p=3$. Since $q \ge 3$ and $p \nmid q$, we have $p < q$ and $pq(q+1) < q^2 (q+1) < |G|_{p'}$. The case of $G = \psl{3}{q}$ with $(3,q-1)=1$ is finished.

(2.2) Now suppose $G = \psl{3}{q}$ with $(3,q-1)=3$. If $G = \psl{3}{4}$, we may use the argument in Remark \ref{remark:2} and compute $d_2(G)=5/64<1/(7-1)=1/(\widetilde{p}-1)$. Thus, we may assume $q \ge 7$ since $(3,q-1)=3$. We have
\begin{linenomath}\begin{align*}
    |G| = \frac 13 q^3 (q-1)^2 (q+1) (q^2+q+1) \quad \text{and} \quad
    k_{p'}(G) < k(G) = \frac{1}{3}(q^2+q+10)<q^2.
\end{align*}\end{linenomath}
So it suffices to show that $p q^2 < |G|_{p'}$. We may assume that $p$ divides the largest factor $q^2+q+1$ of $|G|$. If $p \mid (q-1)$ as well, then $p$ divides $3q=(q^2+q+1)-(q-1)^2$, so $p=3$. Since $q \ge 7$, $q>p$ and $pq^2<q^3<|G|_{p'}$, as desired. On the other hand, if $p \nmid (q-1)$, then $\frac{1}{3}q(q-1)^2$ is a $p'$-number greater than $q^2+q+1\ge p$ since $q \ge 7$. Thus, $p q^2 < \frac{1}{3}q(q-1)^2 q^2 \le |G|_{p'}$ and this case is finished.

(2.3) Now suppose $G = \psu{3}{q}$ with $(3,q+1)=1$. We have
\begin{linenomath}\begin{align*}
    |G| =q^3 (q-1) (q+1)^2 (q^2-q+1) \quad \text{and} \quad
    k_{p'}(G) < k(G) =q^2+q+2 < (q+1)^2.    
\end{align*}\end{linenomath}
So it suffices to show that $p(q+1)^2 < |G|_{p'}$. Assume first that $p \mid (q+1)$. Then $p \nmid q^3(q-1)$ since $p\neq r$ is odd. If $p$ does not divide the largest factor $q^2-q+1$ of $|G|$, then $p(q+1)^2 < (q^2-q+1) q^3 \le |G|_{p'}$ and we're done. On the other hand, if $p \mid (q^2-q+1)$, then $p$ divides $3q=(q+1)^2-(q^2-q+1)$, so $p=3 \le q$. Thus, $p(q+1)^2 \le q(q+1)^2 < (q-1)q^3 \le |G|_{p'}$ for $q \ge 3$, as desired. Finally, assume that $p \nmid q+1$. We may assume that $p$ divides the largest factor $q^2-q+1$ of $|G|$. So we have $p (q+1)^2 \le (q^2-q+1)(q+1)^2 < q^3(q+1)^2 \le |G|_{p'}$ when $q \ge 3$.

(2.4) Finally, suppose $G = \psu{3}{q}$ with $(3,q+1)=3$. We have
\begin{linenomath}\begin{align*}
    |G| = \frac 13 q^3 (q-1) (q+1)^2 (q^2-q+1) \quad \text{and} \quad 
    k_{p'}(G) < k(G) = \frac 13 (q^2+q+12)<  q^2.
\end{align*}\end{linenomath}
So it suffices to show that $pq^2 < |G|_{p'}$. We may assume that $p$ divides the largest factor $q^2-q+1$ of $|G|$. If $p \mid (q+1)$, then $p$ divides $3q=(q+1)^2-(q^2-q+1)$, so $p=3$. Since $q \ge 3$, $pq^2 < q^3 \le |G|_{p'}$, as desired. On the other hand, if $p \nmid (q+1)$, then $\frac 13 q(q+1)^2$ is a $p'$-number greater than $q^2-q+1 \ge p$ when $q \ge 3$. 

We have proved the proposition in the case that $G$ is a simple group $\psl{3}{q}$ or $\psu{3}{q}$.

\vspace{0.5em}

(3) To finish the proof, we consider the simple groups $G = \psl{2}{q}$ with $p \nmid q$. Since $\psl{2}{2}$ and $\psl{2}{3}$ are solvable and $\psl{2}{4} \cong \psl{2}{5} \cong A_5$ was addressed in Proposition \ref{prop:alternating}, we may assume that $q \ge 7$.

(3.1) First assume that $q$ is odd. We have
\begin{linenomath}\begin{align*}
|G|=\frac 12 q(q-1)(q+1) \quad \text{and} \quad k_{p'}(G) < k(G)=(q+5)/2 < q
\end{align*}\end{linenomath}
since $q \ge 7$. So it suffices to show that $pq < |G|_{p'}$. Note that $q-1$ and $q+1$ are even. If $p$ does not divide the largest $p'$-factor $(q+1)/2$ of $|G|$, then $(q+1)/2$ is a $p'$-number greater than $p$ and $pq < \frac 12 (q+1)q \le |G|_{p'}$, as desired. On the other hand, if $p$ divides $(q+1)/2$, then $q-1$ is a $p'$-number greater than $(q+1)/2 \ge p$ since $q \ge 7$ and we have $pq < (q-1)q \le |G|_{p'}$. The case of $q$ odd is finished.

(3.2) Now assume $q$ is even, so we may assume $q \ge 8$. We have
\begin{linenomath}\begin{align*}
|G|=q(q-1)(q+1) \quad \text{and} \quad k_{p'}(G) < k(G) = q+1.
\end{align*}\end{linenomath}
So it suffices to show that $p(q+1) < |G|_{p'}$. If $p \nmid (q+1)$, then $p$ must divide $q-1<q$. Then $p(q+1) <q(q+1) \le |G|_{p'}$, as required. So we may assume $p \mid (q+1)$, and consequently $p \nmid (q-1)$. 

(3.2.i) First assume $q+1$ is not prime. Then $q+1=pm$ for some integer $m>1$. Since $q$ is even, $q+1$ is odd and we have $p = (q+1)/m \le (q+1)/3 < (q-1)/2$. Therefore,
\begin{linenomath}\begin{align*}
    p(q+1) < \frac 12 (q-1) (q+1) < \frac 12 (q-1) (2q) = q(q-1) \le |G|_{p'}. 
\end{align*}\end{linenomath}

(3.2.ii) Finally, assume $q+1=p$ is prime. Then $|G|_{p'}=q(q-1)$ and we compute $d_{p'}(G)$. It is known that $G=\psl{2}{q}$ has the following class structure when $q$ is even: a single class of elements of order 1, a single class of elements of order 2, $(q-2)/2$ classes of nonidentity elements of order dividing $q-1$, and $q/2$ classes of nonidentity elements of order dividing $q+1=p$. Therefore, $k_{p'}(G) = (q+2)/2$. We have
\begin{linenomath}\begin{align*}
    d_{p'}(G) = \frac{q+2}{2q(q-1)} < \frac{1}{q}=\frac{1}{p-1}
\end{align*}\end{linenomath}
since $q \ge 8$. This final case is finished and Proposition \ref{prop:liecrossnowitness} is proved.
\end{proof}

Proposition \ref{prop:psolvableodd_simplegroups} follows from Propositions \ref{prop:alternating}-\ref{prop:liecrossnowitness}. We can now prove the $p$-solvability criterion of Theorem \ref{thm1:odd}.


\begin{theorem}\label{thm:psolvableodd} 
Let $G$ be a finite group and let $p$ be an odd prime. If $d_{p'}(G)>1/(p-1)$, then $G$ is $p$-solvable.
\end{theorem}
\begin{proof}
Assume that $d_{p'}(G)>1/(p-1)$. Suppose for contradiction that $G$ is not $p$-solvable. This means $G$ has a composition factor isomorphic to a nonabelian simple group $S$ with $p$ dividing $|S|$. By applying Lemma \ref{lemma:1}(1) repeatedly to a composition series for $G$, we have $1/(p-1) < d_{p'}(G) \le d_{p'}(S)$. But this contradicts Proposition \ref{prop:psolvableodd_simplegroups}.
\end{proof}

Next, we consider the solvability criterion of Theorem \ref{thm2:even}. Before we begin, recall that the proof of the $p$-solvability criterion in Theorem \ref{thm1:odd} was difficult when the prime $p$ under consideration was the largest prime dividing $|G|$. If $p=2$ and $G$ is a nonabelian finite simple group, this is never the case and there are more applicable results in the literature. Our proof of the solvability criterion in Theorem \ref{thm2:even} uses strong results contained in \cite{M23} for nonabelian finite simple groups, which we collect in the following proposition. We use $\mr{PSL}_n^\epsilon(q)$ to denote $\psl{n}{q}$ when $\epsilon=+$ and $\psu{n}{q}$ when $\epsilon=-$. We also use $\mr{E}^\epsilon_6(q)$ to denote $\mr{E}_6(q)$ when $\epsilon=+$ and ${}^2\mr{E}_6(q)$ when $\epsilon=-$. The statements of Proposition \ref{thm2} are all contained in \cite{M23}. Part (1) is Theorem 1.1, part (2) is Theorem 5.3, part (3) is Theorem 5.4, part (4) is Theorem 5.7, and part (5) is Lemma 5.9 of \cite{M23}.

\begin{proposition}\label{thm2} 
Let $G$ be a finite group and let $\pi$ be a set of primes. 
\begin{enumerate}[nolistsep,label=\textup{(\arabic*)}]
    \item Let $p$ be the smallest member of $\pi$. If $d_\pi(G) > 1/p$, then $G$ has a nilpotent Hall $\pi$-subgroup. Moreover, if $d_\pi(G) > (p^2+p-1)/p^3$, then $G$ has an abelian Hall $\pi$-subgroup.
    \item Let $S$ be a sporadic simple group (including the Tits group) and $\pi=\{p,q\}$ where $p<q$ are odd primes dividing $|S|$. If $(S,\pi) \neq (J_1,\{3,5\})$, then $d_\pi(S) \le 1/p$.
    \item Let $S$ be a finite simple group of Lie type in characteristic $p>2$ and $\pi=\{p,s\}$, where $s$ is an odd prime dividing $|S|$. Then $d_\pi(S) \le 1/s$.
    \item Let $S$ be a finite simple group of Lie type, let $\pi$ be a set of two odd primes not containing the defining characteristic of $S$, and let $p$ be the smaller prime in $\pi$. Assume that we are not in one of the following situations:
    \begin{enumerate}[nolistsep,label=(\roman*)]
        \item $S=\mr{E}^\epsilon_6(q)$ with $3 \in \pi$.
        \item $S=\mr{PSL}_n^\epsilon(q)$ with $n \ge 3$ and $\mr{gcd}(n,q-\epsilon)_\pi \neq 1$.
    \end{enumerate}    
    Then $d_\pi(S) \le 1/p$.
    \item Let $p$ be an odd prime and $S=\mr{PSL}_n^\epsilon(q)$. Assume that $p$ divides $\mr{gcd}(n,q-\epsilon)$ and Sylow $p$-subgroups of $S$ are abelian. Then $n=p=3$.
\end{enumerate}
\end{proposition}

Now we prove the solvability criterion of Theorem \ref{thm2:even}.

\begin{theorem}\label{thm:psolvableeven}
Let $G$ be a finite group. If $d_{2'}(G)>4/15$, then $G$ is solvable.
\end{theorem}

\begin{proof}
The hypothesis is inherited by every composition factor of $G$ by Lemma \ref{lemma:1}(1). So to prove the theorem, it is enough to assume that $G$ is nonabelian simple and prove that $d_{2'}(G) \le 4/15$. Moreover, the proof will show that if $G$ is nonabelian simple, then $d_{2'}(G) = 4/15$ if and only if $G \cong A_5$. By Lemma \ref{lemma:2}, we may assume that $|G|$ is divisible by at least four distinct primes.

(1) Suppose $G$ is a sporadic simple group or the Tits group. Since $|G|$ is divisible by at least four distinct primes, there exist odd primes $p,q$ dividing $|G|$ with $5 \le p < q$. Letting $\pi=\{p,q\} \subseteq 2' \cap \pi(G)$ and using Lemma \ref{lemma:1}(2) and Proposition \ref{thm2}(2), we have $d_{2'}(G) \le d_\pi(G) \le 1/p \le 1/5 < 4/15$. Note that we are not in the exceptional case $(J_1,\{3,5\})$ of Proposition \ref{thm2}(2) since both $p$ and $q$ are greater than or equal to $5$.

(2) Suppose $G$ is a finite simple group of Lie type in characteristic $p>2$. Since $|G|$ is divisible by at least four distinct primes, we can choose a prime $s \neq p$ such that $s \ge 5$. Setting $\pi=\{p,s\}$, Proposition \ref{thm2}(3) says that $d_{2'}(G) \le d_\pi(G) \le 1/s \le 1/5 < 4/15$.

(3) Suppose $G$ is a simple alternating group. Since we addressed $A_5$ and $A_6$ in Lemma \ref{lemma:2}, we may assume that $G = A_n$ with $n \ge 7$. Assume for contradiction that $d_{2'}(G) \ge 4/15$. Let $\pi = \pi(G) - \{2,3\} \subseteq 2'$, so that the smallest prime in $\pi$ is $p=5$. Note that since $n \ge 7$, we have $|\pi|\ge 2$. Now using Lemma \ref{lemma:1}(2), we have $1/p = 1/5 < 4/15 \le d_{2'}(G) \le d_\pi(G)$. Consequently, $G$ has a nilpotent Hall $\pi$-subgroup $H$ by Theorem \ref{thm2}(1). Since $\pi$ consists of odd primes, $H$ is also a nilpotent Hall $\pi$-subgroup of $S_n$. By \cite[Theorem A4]{H56}, the only solvable Hall subgroups of $S_n$ for $n \ge 7$ are the Sylow subgroups and the Hall $\{2,3\}$-subgroups of $S_7$ and $S_8$. Therefore, $G$ has no Hall $\pi$-subgroups since $\pi$ is a collection of at least two odd primes. We have obtained a contradiction, which proves $d_{2'}(A_n) < 4/15$ for $n \ge 7$. Furthermore, Table \ref{table:4primes} shows that $d_{2'}(A_6)<4/15$ and $d_{2'}(A_5)=4/15$.

(4) Suppose $G$ is a finite simple group of Lie type in even characteristic. Since $|G|$ is divisible by at least four distinct primes, there exist two odd primes $5\le p < s$ dividing $|G|$. Let $\pi=\{p,s\} \subseteq 2' \cap \pi(G)$. Note that we are not in case (i) of Proposition \ref{thm2}(4) since $3 \not\in \pi$. Assume for the moment that we are not in case (ii), either. Then by Proposition \ref{thm2}(4), $d_{2'}(G) \le d_\pi(G) \le 1/p \le 1/5 < 4/15$.

To finish the proof, it remains to consider case (ii) of Proposition \ref{thm2}(4). That is, $G = \mathrm{PSL}^\epsilon_n(q)$ with $n \ge 3$, $q=2^a$ and $\mr{gcd}(n,q-\epsilon)_\pi \neq 1$, where $\pi=\{p,s\}$ as before. Suppose for contradiction that $d_{2'}(G) \ge 4/15$. Then since $p \ge 5$ we have
$(p^2+p-1)/p^3 < 4/15 \le d_{2'}(G) \le d_\pi(G)$. Therefore, $G$ has an abelian Hall $\pi$-subgroup by Theorem \ref{thm2}(1). In particular, the Sylow $p$-subgroups and the Sylow $s$-subgroups of $G$ are abelian. Since $p \mid (n,q-\epsilon)$ or $s \mid (n,q-\epsilon)$ by assumption, either $p=3$ or $s=3$ by Theorem \ref{thm2}(5). In either case, this contradicts $5 \le p < s$. Therefore, $d_{2'}(G)<4/15$ in this case as well. 
\end{proof}

\section{A normal Sylow \texorpdfstring{$p$}{p}-subgroup criterion and its consequences}\label{sec:normalp}
In this section, we first prove the normal Sylow $p$-subgroup criterion of Theorem \ref{thm:normalp}, which we reproduce here.
\vspace{0.5em}

\begin{theorem}
Let $G$ be a finite group and let $p$ be a prime. If $d_{p'}(G)>2/(p+1)$, then $G$ has a normal Sylow $p$-subgroup.
\end{theorem}

\begin{proof}
Suppose $G$ is a minimal counterexample. If $p >2$, then $d_{p'}(G) >2/(p+1) \ge 1/(p-1)$ and $G$ is $p$-solvable by Theorem \ref{thm:psolvableodd}. If $p=2$, then $d_{2'}(G) > 2/3 > 4/15$ and $G$ is solvable by Theorem \ref{thm:psolvableeven}. 

We first claim that $\mbf{O}_p(G)=1$. If not, then $|G/\mbf{O}_p(G)|<|G|$ and $d_{p'}(G/\mbf{O}_p(G)) > 2/(p+1)$ by Lemma \ref{lemma:1}(1). Consequently, $G/\mbf{O}_p(G)$ has a normal Sylow $p$-subgroup by the minimality of $|G|$, and then $G$ does as well. This contradicts our assumption that $G$ is a counterexample. Now let $N= \mbf{O}_{p'}(G)$. Since $G$ is $p$-solvable and $\mbf{O}_p(G)=1$, we have $N>1$.

We next claim that $G=NP$, where $P \in \mr{Syl}_p(G)$. Note that $d_{p'}(G/N) > 2/(p+1)$ by Lemma \ref{lemma:1}(1), so $NP/N \nrm G/N$ by the minimality of $|G|$. In particular, $NP \nrm G$ and so $d_{p'}(NP) > 2/(p+1)$ by Lemma \ref{lemma:1}(1) again. If $NP < G$, then $P \nrm NP$ by the minimality of $|G|$. But then $P$ is characteristic in $NP \nrm G$ and so $P \nrm G$, a contradiction. So $G=NP$ and we have $\Norm{G}{P} = \Norm{G}{P} \cap NP = \Cen{N}{P} P$ by Dedekind's modular law, which implies $[N:\Cen{N}{P}]=[G:\Norm{G}{P}] \ge p+1$ by Sylow's theorem.

We reach our final contradiction by bounding $k_{p'}(G)$ from above. The $p'$-conjugacy classes of $G=NP$ are the orbits of $G$ acting on $N$ by conjugation. We obtain an upper bound for $k_{p'}(G)$ if we count just the $P$-orbits. The number of trivial $P$-orbits is $|\Cen{N}{P}|$, and each nontrivial $P$-orbit has size at least $p$. Therefore,
\begin{linenomath}\begin{align*}
    k_{p'}(G) &\le \frac{|N|-|\Cen{N}{P}|}{p} + |\Cen{N}{P}| 
    = \frac{1}{p}|N| + \frac{p-1}{p} |\Cen{N}{P}|,
\end{align*}\end{linenomath}
and then
\begin{linenomath}\begin{align*}
    \frac{2}{p+1} &< d_{p'}(G) \le \frac 1p + \frac{p-1}{p} \frac{1}{[N:\Cen{N}{P}]} \le \frac 1p + \frac{p-1}{p(p+1)} = \frac{2}{p+1},
\end{align*}\end{linenomath}
a contradiction.
\end{proof}

Next, we use the normal Sylow $p$-subgroup criterion of Theorem \ref{thm:normalp} to refine our knowledge of the groups occurring in Theorems \ref{thm1:odd} and \ref{thm2:even}. We begin by bounding the $p$-length of the $p$-solvable groups.

\begin{proposition}\label{prop:6.1}
Let $G$ be a finite group and let $p$ be an odd prime. If $d_{p'}(G)>1/(p-1)$, then $G$ is $p$-solvable and has $p$-length at most 2.
\end{proposition}

\begin{proof}
Let $G$ be a finite group, $p$ an odd prime, and assume $d_{p'}(G)>1/(p-1)$. By Theorem \ref{thm:psolvableodd}, $G$ is $p$-solvable. Let $M_0=1$ and define the characteristic subgroup $M_{i+1}$ of $G$ recursively by $M_{i+1}/M_i=\mbf{O}_{p',p}(G/M_i)$ for $i \ge 0$. Recall that $\mbf{O}_{p'}(G)$ is the unique maximal normal $p'$-subgroup of $G$, and $\mbf{O}_{p',p}(G) \nrm G$ is defined by $\mbf{O}_{p',p}(G)/\mbf{O}_{p'}(G) = \mbf{O}_p(G/\mbf{O}_{p'}(G))$. Let $k$ be the smallest integer such that $M_k=G$, which exists since $G$ is $p$-solvable. The $p$-length of $G$ is $k$ or $k-1$, depending on whether $M_k/M_{k-1}$ contains a nontrivial $p$-group. Although it is possible that $\mbf{O}_p(G)>1$, it is an important observation for our proof that $\mbf{O}_p(G/M_i)=1$ for all $i \ge 1$ by construction of $M_i$. Since $\mbf{O}_p(M_{i+1}/M_i)$ is characteristic in $M_{i+1}/M_i$, which is itself characteristic in $G/M_i$, we have $\mbf{O}_p(M_{i+1}/M_i) \le \mbf{O}_p(G/M_i)=1$ for $i \ge 1$. In particular, if $M_{i+1}/M_i$ has a nontrivial Sylow $p$-subgroup for $i \ge 1$, then it is not normal.

Suppose for contradiction that $G$ has $p$-length greater than 2. Then $M_{i+1}/M_i$ has a nontrivial Sylow $p$-subgroup for $0 \le i \le 2$ and it is not normal for $1 \le i \le 2$. Therefore, $d_{p'}(M_{i+1}/M_i) \le 2/(p+1)$ for $1 \le i \le 2$ by Theorem \ref{thm:normalp}. By applying Lemma \ref{lemma:1}(1) repeatedly to the normal series $M_0 \nrm M_1 \nrm \dots \nrm M_k=G$, we have
\begin{linenomath}\begin{align*}
    \frac{1}{p-1} < d_{p'}(G) \le \prod_{i=0}^{k-1} d_{p'}(M_{i+1}/M_i) \le \prod_{i=1}^{2} d_{p'}(M_{i+1}/M_i) \le \bigg( \frac{2}{p+1} \bigg)^2.
\end{align*}\end{linenomath}
But this is a contradiction.
\end{proof}

\begin{proposition}\label{prop:6.2}
Let $G$ be a finite group. If $d_{2'}(G)>4/15$, then $G$ is solvable and has $2$-length at most 4.
\end{proposition}
\begin{proof}
Let $G$ be a finite group and assume $d_{2'}(G)>4/15$. By Theorem \ref{thm:psolvableeven}, $G$ is solvable. Assume for contradiction that $G$ has $2$-length greater than 4. Again, let $M_0=1$ and construct the characteristic subgroup $M_{i+1}$ recursively by $M_{i+1}/M_i=\mbf{O}_{2',2}(G/M_i)$ for $i \ge 0$. Then $M_{i+1}/M_i$ has a nontrivial, nonnormal Sylow $2$-subgroup for $1 \le i \le 4$. Therefore, $d_{p'}(M_{i+1}/M_i) \le 2/3$ for $1 \le i \le 4$ by Theorem \ref{thm:normalp}. In the same way as before, we have
\begin{linenomath}\begin{align*}
    \frac{4}{15} < d_{2'}(G) \le \prod_{i=1}^4 d_{p'}(M_{i+1}/M_i) \le \bigg( \frac{2}{3} \bigg)^4.
\end{align*}\end{linenomath}
This is again a contradiction.
\end{proof}

A similar proof strategy allows us to prove Theorem \ref{thm4:plength}.

\begin{proposition}
Let $G$ be a finite group and let $p$ be a prime. If $G$ is $p$-solvable of $p$-length $k$, then $d_{p'}(G) \le \big( 2/(p+1) \big)^{k-1}$.
\end{proposition}

\begin{proof}
Once again, let $M_0=1$ and define the characteristic subgroup $M_{i+1}$ of $G$ recursively by $M_{i+1}/M_{i}=\mbf{O}_{p',p}(G/M_{i})$ for $i \ge 0$. Since $G$ is $p$-solvable of $p$-length $k$, the quotient $G/M_k$ is a (possibly trivial) $p'$-group. The quotient $M_{i+1}/M_{i}$ has a nontrivial, nonnormal Sylow $p$-subgroup for $1 \le i \le k-1$. Therefore, $d_{p'}(M_{i+1}/M_{i}) \le 2/(p+1)$ for $1 \le i \le k-1$ by Theorem \ref{thm:normalp}. In the same way as before, we have
\begin{align*}
    d_{p'}(G) \le \prod_{i=1}^{k-1} d_{p'}(M_{i+1}/M_{i}) \le \bigg( \frac{2}{p+1} \bigg)^{k-1}.
\end{align*}
\end{proof}

Next, we classify those groups saturating the bound in Theorems \ref{thm1:odd} and \ref{thm2:even}.

\begin{proposition}\label{prop:6.3}
Let $G$ be a finite group and let $p$ be an odd prime. Assume $G$ is not $p$-solvable. Let $H=G/\mbf{O}_p(G)$ and $Z=\Zen{H}$. Then $d_{p'}(G)=1/(p-1)$ if and only if $p>3$, $H/Z$ is isomorphic to $\psl{2}{p}$ or $\pgl{2}{p}$, and $H' \cong \psl{2}{p}$.
\end{proposition}

\begin{proof}
Suppose $G$ is not $p$-solvable and $d_{p'}(G)=1/(p-1)$. We will use the following observation repeatedly.

(1) If $N \nrm M \nrm G$ and $M/N$ is not $p$-solvable, then $d_{p'}(M/N)=1/(p-1)$.

Using Lemma \ref{lemma:1}(1), we have
\begin{linenomath}\begin{align*}
    \frac{1}{p-1} = d_{p'}(G) \le d_{p'}(M/N) \le \frac{1}{p-1},
\end{align*}\end{linenomath}
where the last inequality follows from Theorem \ref{thm:psolvableodd} since $M/N$ is not $p$-solvable. So we must have equality throughout and the claim is proved.

(2) Note that $G/\mbf{O}_p(G)$ is not $p$-solvable, since otherwise $G$ is $p$-solvable. Then $d_{p'}(G/\mbf{O}_p(G))=1/(p-1)$ by (1), so $G/\mbf{O}_p(G)$ satisfies our hypothesis. In what follows, we work in the quotient $G/\mbf{O}_p(G)$ and assume without loss of generality that $\mbf{O}_p(G)=1$.

(3) Let $Z$ be the $p$-solvable radical of $G$; that is, the largest normal $p$-solvable subgroup of $G$. Since $G$ is not $p$-solvable, $G/Z>1$. Using Lemma \ref{lemma:1}(1), we have
\begin{linenomath}\begin{align*}
    \frac{1}{p-1} = d_{p'}(G) \le d_{p'}(Z) \, d_{p'}(G/Z) \le d_{p'}(G/Z) = \frac{1}{p-1},
\end{align*}\end{linenomath}
where the last equality follows from (1). This forces $d_{p'}(Z)=1>2/(p+1)$, so $Z$ has a normal Sylow $p$-subgroup $P$ by Theorem \ref{thm:normalp}. But then $P$ is characteristic in $Z \nrm G$ and consequently $P \le \mbf{O}_p(G)=1$. So $Z$ is a $p'$-subgroup with $1=d_{p'}(Z)=d(Z)$, which means $Z$ is abelian by Gustafson \cite{G73}.

(4) Let $\q{G}=G/Z$ and use the `bar' notation. Let $\q{M}$ be a minimal normal subgroup of $\q{G}$. Then $\q{M} \cong S^n$ for some nonabelian simple group $S$ with order divisible by $p$ and some natural number $n$. By (1), $1/(p-1)=d_{p'}(\q{M}) = d_{p'}(S)^n$. Since $d_{p'}(S) \le 1/(p-1)$ by Theorem \ref{thm:psolvableodd}, we have $n=1$ and $d_{p'}(S) = 1/(p-1)$. Then by Proposition \ref{prop:psolvableodd_simplegroups}, $p>3$ and $S \cong \psl{2}{p}$. Now, $\q{C} = \Cen{\q{G}}{\q{M}} \nrm \q{G}$ and $d_{p'}(\q{C} \,\q{M})=1/(p-1)$ by (1). Since $d_{p'}(\q{M})=1/(p-1)$ by (1) again, $1/(p-1) \le d_{p'}(\q{C}) d_{p'}(\q{M}) \le d_{p'}(\q{M})=1/(p-1)$, which forces $d_{p'}(\q{C})=1$. Thus, $\q{C}$ has a normal Sylow $p$-subgroup by Theorem \ref{thm:normalp}. In particular, $\q{C}$ is $p$-solvable, so $\q{C}=\q{1}$. We have shown that $\q{G}$ is an almost simple group with socle $\psl{2}{p}$, so $\q{G}$ is isomorphic to $\psl{2}{p}$ or $\pgl{2}{p}$. It is well-known that $k_{p'}(\pgl{2}{p})=p+1$, so $d_{p'}(\pgl{2}{p})=1/(p-1)$ and this case can occur.

(5) We claim that $Z$ is a central subgroup of $G$. We showed in (3) that $Z$ is an abelian $p'$-group. By (4), $|\q{G}|_{p'}=\epsilon (p-1)(p+1)$, where $\epsilon = \tfrac 12$ if $\q{G} \cong \psl{2}{p}$ and $\epsilon=1$ if $\q{G} \cong \pgl{2}{p}$. Then
\begin{linenomath}\begin{align*}
    \frac{1}{p-1} = d_{p'}(G) = \frac{k_{p'}(G)}{|Z| \epsilon (p-1)(p+1)},
\end{align*} \end{linenomath}
and it follows that $k_{p'}(G) = |Z| \epsilon (p+1) = k_{p'}(Z) k_{p'}(\q{G})$. Now, if $x_1, \dots, x_{\epsilon(p+1)}$ are elements of $G$ such that the $\q{x}_i$ are representatives of the $p'$-conjugacy classes of $\q{G}$, then every $p'$-element of $G$ lies in a conjugacy class of the form $(x_i z)^G$ with $i=1,\dots, \epsilon(p+1)$ and $z \in Z$. By our formula for $k_{p'}(G)$, all of these classes must be $p'$-classes and they must be distinct. In particular, $z^G=\{z\}$ for all $z \in Z$. That is, $Z \le \Zen{G}$. In fact, $Z = \Zen{G}$ since $\Zen{\q{G}}$ is trivial.

(6) By what we proved in (4) and (5), $Z=\Zen{G}$ and $G/Z$ is isomorphic to $\psl{2}{p}$ or $\pgl{2}{p}$. This means $G'Z/Z = (G/Z)' \cong \psl{2}{p}$. Let $N=G^{(\infty)} \nrm G'$, the last term of the derived series for $G$. Note that $N \cap Z < N$, since otherwise $G$ is solvable. We have $1 < N/N\cap Z \cong NZ/Z \nrm G'Z/Z \cong \psl{2}{p}$. Since $\psl{2}{p}$ is simple for $p>3$, $NZ/Z=G'Z/Z$ and $N/N \cap Z \cong \psl{2}{p}$. That is, $N$ is a perfect central extension of $\psl{2}{p}$. If $N\cap Z>1$, then $N \cong \sl{2}{p}$, the Schur cover of $\psl{2}{p}$ when $p > 3$. But then $N \nrm G$ and
\begin{linenomath}\begin{align*}
    d_{p'}(N) = d_{p'}(\sl{2}{p}) = \frac{p}{(p-1)(p+1)} \neq \frac{1}{p-1},
\end{align*}\end{linenomath}
which contradicts (1). So $N \cap Z = 1$ and $N \cong \psl{2}{p}$. In the quotient $G/N$, $NZ/N \le \Zen{G/N}$ and $[G:NZ]=[G:G'Z] \le 2$ means $(G/N)/(NZ/N)$ is cyclic. Therefore, $G/N$ is abelian and $G' \le N$. The reverse conclusion is obvious, so $G' = N \cong \psl{2}{p}$.  We have proved one direction of the proposition.

For the reverse direction, first note that $\mbf{O}_p(G)$ is contained in the kernel of every irreducible $p$-modular representation \cite[Lemma 2.3]{N98}, so $d_{p'}(G) = d_{p'}(G/\mbf{O}_p(G))$. Thus, we may assume without loss of generality that $\mbf{O}_p(G)=1$. Then $Z=\Zen{G}$ is a $p'$-group and $G' \cap Z=1$. Let $\q{G}=G/Z$ and use the `bar' notation. Let $\q{x} \in \q{G}$ be a representative of one of the $\epsilon(p+1)$ $p$-regular conjugacy classes $\q{x}^{\q{G}}$ of $\q{G}$, where again $\epsilon = \tfrac 12$ if $\q{G} \cong \psl{2}{p}$ and $\epsilon=1$ if $\q{G} \cong \pgl{2}{p}$. The preimage of $\q{x}^{\q{G}}$ as a subset of $G$ is $\bigcup_{z \in Z} (xz)^G$. Note that $(xz)^G$ is a $p$-regular class of $G$ for every $z \in Z$ since $Z$ is a $p'$-group. Suppose that $(xz)^G = (xz')^G$ for some $z,z' \in Z$. Then $x z' = (xz)^g = x^g z$ for some $g \in G$. But then $z' z^{-1} = [x,g] \in G' \cap Z = 1$, so $z=z'$. Therefore, the union is disjoint and every $p$-regular class of the quotient $\q{G}$ lifts to $|Z|$ $p$-regular classes of $G$, so $k_{p'}(G) = |Z| \epsilon (p+1)$ and $d_{p'}(G) = 1/(p-1)$, as claimed. 
\end{proof}

\begin{proposition}\label{prop:6.4}
Let $G$ be a finite group. Assume $G$ is nonsolvable. Let $H=G/\mbf{O}_2(G)$ and $Z=\Zen{H}$. Then $d_{2'}(G)=4/15$ if and only if $H \cong A_5 \times Z$.
\end{proposition}

\begin{proof}
The proof follows in the same fashion as that of Proposition \ref{prop:6.3}, so we give a sketch of the proof. For both directions, we may assume without loss of generality that $\mbf{O}_2(G)=1$. Then $Z=\Zen{G}$ is a $2'$-group. By the proof of Theorem \ref{thm:psolvableeven}, the only nonabelian finite simple group $S$ satisfying $d_{2'}(S)=4/15$ is $S \cong A_5$, so $G/Z$ is an almost simple group with socle $A_5$. Since $S_5 = \mr{Aut}(A_5)$ and $d_{2'}(S_5)=1/5 \neq 4/15$, we must have $G/Z \cong A_5$. The Schur multiplier of $A_5$ is cyclic of order $2$ and $Z$ is a $2'$-group, so the central extension must split and $G \cong A_5 \times Z$. For the reverse direction, $d_{p'}(G)=d_{p'}(Z \times A_5) = 4/15$.
\end{proof}

Finally, we obtain some information about the $p'$-structure of the groups in Theorem \ref{thm1:odd}.

\begin{proposition}\label{prop:6.5}
Let $G$ be a finite group and let $p$ be an odd prime. Suppose $d_{p'}(G)>1/(p-1)$. Then $G$ is $p$-solvable and the number of nonabelian simple $p'$-factors in a composition series for $G$ is strictly less than $\ln(p-1)/\ln(12)$. In particular, if $p \le 13$, then $G$ is solvable.
\end{proposition}

\begin{proof}
Let $\{S_i\}_{i=1}^r$ be the set of composition factors of $G$. Since $G$ is $p$-solvable by Theorem \ref{thm:psolvableodd}, each $S_i$ is either a simple $p'$-group or isomorphic to $C_p$. Let $J \subseteq \{1,\dots,r\}$ be the largest index set such that $S_j$ is a nonabelian simple $p'$-group for all $j \in J$. By applying Lemma \ref{lemma:1}(1) repeatedly to a composition series for $G$, we have
\begin{linenomath}\begin{align*}
    \frac{1}{p-1} &< d_{p'}(G) \le \prod_{i=1}^r d_{p'}(S_i) \le \prod_{j \in J} d_{p'}(S_j) \le \prod_{j \in J} d(S_j),
\end{align*}\end{linenomath}
where $d_{p'}(S_j)=d(S_j)$ for all $j\in J$ since the $S_j$ are $p'$-groups. Now, because the $S_j$ are nonabelian simple groups, $d(S_j) \le 1/12$ for all $j \in J$ by a result of Dixon \cite{D73}. Therefore,
\begin{linenomath}\begin{align*}
    \frac{1}{p-1} < \bigg( \frac{1}{12} \bigg)^{|J|}.
\end{align*}\end{linenomath}
Rearranging this inequality yields $|J| < \ln(p-1)/\ln(12)$, as claimed. 
\end{proof}

We have proved Theorem \ref{thm1:odd}: Part (1) is Theorem \ref{thm:psolvableodd}, Proposition \ref{prop:6.1} and Proposition \ref{prop:6.5}. Part (2) is Proposition \ref{prop:6.3}. We have also proved Theorem \ref{thm2:even}: Part (1) is Theorem \ref{thm:psolvableeven} and Proposition \ref{prop:6.2}. Part (2) is Proposition \ref{prop:6.4}.

The results in this section make clear that Theorems \ref{thm1:odd} and \ref{thm2:even} are intimately related to Theorem \ref{thm:normalp}. It would be of interest to study the finite groups $G$ satisfying $d_{p'}(G) \in (1/(p-1),2/(p+1)]$ for an odd prime $p$. As a first step in this direction, we note that for every $p>3$ there exists a finite group $G$ such that $d_{p'}(G)$ lands in this interval (as long as some number theoretical requirements hold). Fixing $p>3$, for every prime $r$ and integer $a$ such that $p\mid (r^a-1)$ and $p+1 \le r^a < (p-1)^2$, there exists a group $G$ of order $r^a p$ such that $1/(p-1) < d_{p'}(G) \le 2/(p+1)$. The number theoretical requirement for the existence of such a prime power $r^a$ is equivalent to the existence of an integer $k$ with $1 \le k < p-2$ such that $pk+1$ is a prime power. It is expected that such a $k$ exists for every odd prime $p>3$, but it has not been proved.

Finally, we conclude this section by showing that the theorem of Sangroniz we mentioned in the introduction cannot be generalized to arbitrary finite groups. These examples were provided to us by G. Navarro. Take $G=2.A_5.2$, $H$ a maximal subgroup isomorphic to $C_5:C_8$, and $p=2$. Then the three irreducible 2-Brauer characters of $G$ restrict irreducibly to $H$, but $d_{2'}(G)=1/5 \neq 2/5 = d_{2'}(H)$. For the other direction, take $G=A_9$, $H=A_8$ and $p=3$. Then $d_{3'}(G)=d_{3'}(H)$ but not every irreducible 3-Brauer character of $G$ restricts irreducibly to $H$. For example, $G$ has an irreducible 3-Brauer character of degree 189, but $H$ does not.

\section{GAP Calculations}\label{sec:calculations}

\subsection{GAP Code}
This section contains GAP functions for calculating $d_2(G)$ and $d_{2'}(G)$ for a nonabelian finite simple group $G$. The function \verb|d2G(G)| returns $d_2(G)$, $1/(\widetilde{p}-1)$ and the results of the logical test $d_2(G) \le 1/(\widetilde{p}-1)$.

\begin{lstlisting}[language=Python]
d2G:=function(G)
    local classorders,num2class,2part,primesinG,plarge,d2,bound;
    classorders:=List(List(ConjugacyClasses(G),Representative),Order); # Orders of class representatives.
    num2class:=Size(Filtered(Filtered(classorders,IsPrimePowerInt),IsEvenInt))+1; # Number of 2-classes
    2part:=2.^PrimePowersInt(Size(G))[2]; # Requires |G| even
    primesinG:=PrimeDivisors(Size(G));
    plarge:=Maximum(primesinG);
    d2:=num2class/2part;
    bound:=1./(plarge-1);
    return [d2,bound,d2 <= bound];
end;
\end{lstlisting}
The function \verb|d2prime| returns $d_{2'}(G)$.
\begin{lstlisting}
d2prime:=function(G)
    local classorders,numoddclass,SizeG,2part,oddpart;
    classorders:=List(List(ConjugacyClasses(G),Representative),Order); # Orders of class representatives.
    numoddclass:=Size(Filtered(classorders,IsOddInt));
    SizeG:=Size(G);
    2part:=2^PrimePowersInt(SizeG)[2]; # Requires |G| even
    oddpart:=SizeG/2part;
    return numoddclass/oddpart;
end;
\end{lstlisting}

Here are some examples of these functions in action. 

\begin{lstlisting}
d2G(AlternatingGroup(9));
[ 0.078125, 0.166667, true ]
d2G(SimpleGroup("Suz"));
[ 0.0012207, 0.0833333, true ]
# Examples from Lemma 2.2.
Lemma2Examples:=[PSL(2,5),PSL(2,9),PSU(4,2),PSL(3,2),PSL(2,8),PSU(3,3),PSL(3,3),PSL(2,17)];;
List(Lemma2Examples,d2prime);
[ 4/15, 1/9, 8/405, 4/21, 8/63, 5/189, 7/351, 7/153 ]
\end{lstlisting}

\subsection{Calculations for small alternating groups.} 
In this section, we show that $d_{p'}(A_n) \le 1/(p-1)$ for $5 \le n <10$, with equality if and only if $p=5=n$. For $n=8,9$, we compute using Remark \ref{remark:2}.

\tiny
{\renewcommand{\arraystretch}{1.2}
\begin{table}[ht]
\centering
    \begin{tabular}{c|c|c|c|c}
\hline
$n$ 
& $k_2(A_n)$ 
& $|A_n|_2$ 
& $\widetilde{p}$ 
& $d_2(G) < 1/(\widetilde{p}-1)$ \\ \hline
9 & 5 & 64 & 7 & 0.07813 $<$ 0.16667 \checkmark \\
8 & 5 & 64 & 7 & 0.07813 $<$ 0.16667 \checkmark \\ \hline
\end{tabular}
\end{table}
\normalsize

We check the remaining cases by considering each odd prime $p$ dividing the order. 

\tiny
{\renewcommand{\arraystretch}{1.2}
\begin{longtable}{c|c|c|c|c}
\hline
$n$ 
& $p$ 
& $k_{p'}(A_n)$ 
& $|A_n|_{p'}$ 
& $d_{p'}(A_n) \le 1/(p-1)$ \\ \hline

7 & 7 & 7 & 360 & 0.01944 $<$ 0.16667 \checkmark \\
  & 5 & 8 & 504 & 0.01587 $<$ 0.25000 \checkmark\\
  & 3 & 6 & 280 & 0.02143 $<$ 0.50000 \checkmark  \\
6 & 5 & 5 & 72 & 0.06944 $<$ 0.25000 \checkmark \\
  & 3 & 5 & 40 & 0.12500 $<$ 0.50000 \checkmark\\
5 & 5 & 3 & 12 & 0.25000 = 0.25000 \checkmark  \\
  & 3 & 4 & 20 & 0.20000 $<$ 0.50000 \checkmark \\ \hline
\end{longtable}
\normalsize

\subsection{Calculations for sporadic groups and the Tits group.} 
In this section, we prove Proposition \ref{prop:sporadic}. When $|G|$ is small, we calculate using our GAP code and the computational strategy described in Remark \ref{remark:2}. When $|G|$ is large, we can reason with the even coarser estimate $d_2(G) < k(G)/|G|_2 < 1/(\widetilde{p}-1)$. In some cases, we extract the classes from the ATLAS \cite{ATLAS} using a GAP command like \verb|AtlasClassNames(CharacterTable("J3"));| and calculate by hand. 

\tiny
{\renewcommand{\arraystretch}{1.2}
\begin{longtable}{c|c|c|c|c}
\hline
$G$
& $\#$ of conjugacy classes 
& $|G|_2$ 
& $\widetilde{p}$ 
& Compare $k(G)/|G|_2$ or $k_2(G)/|G|_2$ to $1/(\widetilde{p}-1)$ \\ \hline

$M_{11}$ & $k_2(G)=5$  & $2^4$ & 11 & 0.31250 $\not< 0.10000$ \quad \texttimes \\
$M_{12}$ & $k_2(G)=7$  & $2^6$ & 11 &  0.10938 $\not<$ 0.10000 \quad \texttimes \\
$M_{22}$ & $k_2(G)=5$  & $2^7$ & 11 & 0.03906 $<$ 0.10000 \quad \checkmark \\
$M_{23}$ & $k_2(G)=4$  & $2^7$ & 23 & 0.03125 $<$ 0.04545 \quad \checkmark \\
$M_{24}$ & $k_2(G)=7$  & $2^{10}$ & 23 & 0.00684 $<$ 0.04545 \quad \checkmark \\
${}^2 F_4(2)'$ & $k(G)=22$  & $2^{11}$ & 13 & $0.01074 < 0.08333$ \quad \checkmark \\
$J_1$ & $k_2(G)=2$  & $2^3$ & 19 & 0.25000 $\not<$ 0.05556 \quad \texttimes \\
$J_2$ & $k_2(G)=5$  & $2^7$ & 7 & 0.03906 $<$ 0.16667 \quad \checkmark \\
$J_3$ & $k_2(G)=4$  & $2^7$ & 19 & 0.03125 $<$ 0.05556 \quad \checkmark \\
$J_4$ & $k(G)=62$  & $2^{21}$ & 43 & $0.00003 < 0.02381$ \quad \checkmark \\
$Co_1$ & $k(G)=101$  & $2^{21}$ & 23 & $0.00005 <$ 0.04545 \quad \checkmark \\
$Co_2$ & $k(G)=60$  & $2^{18}$ & 23 & $0.00023<$ 0.04545 \quad \checkmark \\
$Co_3$ & $k(G)=42$  & $2^{10}$ & 23 & $0.04102<$ 0.04545 \quad \checkmark \\
$Fi_{22}$ & $k(G)=65$  & $2^{17}$ & 13 & $0.00050<$ 0.08333 \quad \checkmark \\
$Fi_{23}$ & $k(G)=98$  & $2^{18}$ & 23 & $0.00037<$ 0.04545 \quad \checkmark \\
$Fi_{24}'$ & $k(G)=108$  & $2^{21}$ & 29 & $0.00005<$ 0.03571 \quad \checkmark \\
$HS$ & $k_2(G)=9$  & $2^9$ & 11 & 0.01758 $<$ 0.10000 \quad \checkmark \\
$McL$ & $k_2(G)=4$  & $2^7$ & 11 & 0.03125 $<$ 0.10000 \quad \checkmark \\
$He$ & $k_2(G)=7$  & $2^{10}$ & 17 & 0.00684 $<$ 0.06250 \quad \checkmark \\
$Ru$ & $k(G)=36$  & $2^{14}$ & 29 & 0.00220 $<$ 0.03571 \quad \checkmark \\
$Suz$ & $k_2(G)=10$  & $2^{13}$ & 13 & 0.00122 $<$ 0.08333 \quad \checkmark \\
$O'N$ & $k(G)=30$  & $2^9$ & 31 & 0.05859 $<$ 0.03333 \quad \checkmark \\
$HN$ & $k(G)=54$  & $2^{14}$ & 19 & $0.00330<$ 0.05556 \quad \checkmark \\
$Ly$ & $k_2(G)=5$  & $2^8$ & 67 & 0.01953 $\not<$ 0.01515 \quad \texttimes \\
$Th$ & $k(G)=48$ & $2^{15}$ & 31 & $0.00146<$ 0.03333 \quad \checkmark \\
$B$ & $k(G)=184$   & $2^{41}$ & 47 & $\approx 10^{-10}<$ 0.02174 \quad \checkmark \\
$M$ & $k(G)=194$  & $2^{46}$ & 71 & $\approx 10^{-12}<$ 0.01429 \quad \checkmark \\ \hline
\end{longtable}
\normalsize

To finish, we show that if $G$ is isomorphic to one of $M_{11}$, $M_{12}$, $J_1$, or $Ly$, then $d_{p'}(G) < 1/(p-1)$ for every odd prime $p$ dividing $G$.

\tiny
{\renewcommand{\arraystretch}{1.2}
\begin{longtable}{c|c|c|c|c}
\hline
$G$ 
& $p$ 
& $k_{p'}(G)$
& $|G|_{p'}$  
& $d_{p'}(G) < 1/(p-1)$ \\ \hline

$M_{11}$ & 3 & $8$ & $880$ & 0.00909 $<$ 0.50000 \checkmark  \\

& 5 & $9$ & $1584$ & 0.00568 $<$ 0.25000 \checkmark \\

& 11 & $8$ & $720$ & 0.01111 $<$ 0.10000 \checkmark \\[1em]

$M_{12}$ & 3 & $11$ & $3520$ & 0.00313 $<$ 0.50000 \checkmark \\
& 5 & $13$ & $19008$ & 0.00068 $<$ 0.25000 \checkmark\\
& 11 & $13$ & $8640$ & 0.00150 $<$ 0.10000 \checkmark \\[1em]

$J_1$ & 3 & $11$ & $58520$ & 0.00019 $<$ 0.50000 \checkmark \\
& 5 & $9$ & $35112$ & 0.00026 $<$ 0.25000 \checkmark \\
& 7 & $14$ & $25080$ & 0.00056 $<$ 0.16667 \checkmark \\
& 11 & $14$ & $15960$ & 0.00088 $<$ 0.10000 \checkmark \\
& 19 & $12$ & $9240$ & 0.00130 $<$ 0.05556 \checkmark \\[1em]

$Ly$ & 3 & $30$ & $\approx 10^{13}$ & $\approx 10^{-12}$ $<$ 0.50000 \checkmark  \\
& 5 & $40$ & $\approx 10^{12}$ & $\approx 10^{-11}$ $<$ 0.25000 \checkmark \\
& 7 & $46$ & $\approx 10^{16}$ & $\approx 10^{-15}$ $<$ 0.16667 \checkmark \\
& 11 & $47$ & $\approx 10^{15}$ & $\approx 10^{-14}$ $<$ 0.10000 \checkmark \\
& 31 & $48$ & $\approx 10^{15}$ & $\approx 10^{-14}$ $<$ 0.03333 \checkmark \\
& 37 & $51$ & $\approx  10^{15}$ & $\approx 10^{-13}$ $<$ 0.02778 \checkmark \\
& 67 & $50$ & $\approx 10^{15}$ & $\approx 10^{-13}$ $<$ 0.01515 \checkmark \\ \hline
\end{longtable}
\normalsize

\, \\
\noindent \textbf{Acknowledgements.} The author is grateful to Professor Hung Tong-Viet for supervising this project as part of his PhD studies. He thanks the referees for their careful reading of the paper and their comments/suggestions, as well as Attila Mar{\'o}ti and Nguyen N. Hung for reading a previous version. He also thanks Thomas Keller for asking about Theorem \ref{thm4:plength}.


\bigskip \footnotesize

\textsc{Department of Mathematics and Statistics, Binghamton University,
    Binghamton, NY 13902-6000, USA}\par\nopagebreak
  \textit{E-mail address}: \texttt{schroedc@math.binghamton.edu}
  

\begin{thebibliography}{10}

\bibitem{ATLAS} J. H. Conway, R. T. Curtis, S. P. Norton, R. A. Parker, and R. A. Wilson, Atlas of finite groups, Maximal subgroups and ordinary characters for simple groups, With computational assistance from J. G. Thackray, Oxford University Press, Eynsham, 1985.

\bibitem{B22} T.~C. Burness, R. Guralnick, A. Moret\'{o}\ and G. Navarro, On the commuting probability of $p$-elements in a finite group, Algebra Number Theory {\bf 17} (2023), no.~6, 1209--1229. 

\bibitem{D73} J.~D. Dixon, Solution to Problem 176, Canadian Mathematical Bulletin, {\bf 16} (1973), 302.

\bibitem{E15} S. Eberhard, Commuting probabilities of finite groups, Bull. Lond. Math. Soc. {\bf 47} (2015), no.~5, 796--808. 

\bibitem{F12} J. Fulman\ and\ R.~M. Guralnick, Bounds on the number and sizes of conjugacy classes in finite Chevalley groups with applications to derangements, Trans. Amer. Math. Soc. {\bf 364} (2012), no.~6, 3023--3070. 

\bibitem{GAP4} The GAP~Group, \emph{GAP -- Groups, Algorithms, and Programming, Version 4.12.2}; 2022, \url{https://www.gap-system.org}.

\bibitem{G06} R.~M. Guralnick\ and\ G.~R. Robinson, On the commuting probability in finite groups, J. Algebra {\bf 300} (2006), no.~2, 509--528. 

\bibitem{G73} W.~H. Gustafson, What is the probability that two group elements commute?, Amer. Math. Monthly {\bf 80} (1973), 1031--1034. 

\bibitem{H56} P. Hall, Theorems like Sylow's, Proc. London Math. Soc. (3) {\bf 6} (1956), 286--304. 

\bibitem{M14} N.~N. Hung\ and\ A. Mar\'{o}ti, On the number of conjugacy classes of $\pi$-elements in finite groups, Arch. Math. (Basel) {\bf 102} (2014), no.~2, 101--108. 

\bibitem{M23} N.~N. Hung, A. Mar\'{o}ti\ and\ J. Mart\'{\i}nez, Conjugacy classes of $\pi$-elements and nilpotent/abelian Hall $\pi$-subgroups, Pacific J. Math. {\bf 323} (2023), no.~1, 185--204. 

\bibitem{H00} B. Huppert\ and\ W. Lempken, Simple groups of order divisible by at most four primes,
{\it Proc. F. Scorina Gemel State Univ.} {\bf 16} (2000), 64--75. 

\bibitem{L95} P. Lescot, Isoclinism classes and commutativity degrees of finite groups, J. Algebra {\bf 177} (1995), no.~3, 847--869. 

\bibitem{M21} G. Malle, G. Navarro\ and\ G.~R. Robinson, Conjugacy class numbers and $\pi$-subgroups, Pacific J. Math. {\bf 311} (2021), no.~1, 135--164. 

\bibitem{M11} G. Malle\ and\ D.~M. Testerman, {\it Linear algebraic groups and finite groups of Lie type}, Cambridge Studies in Advanced Mathematics, 133, Cambridge Univ. Press, Cambridge, 2011. 

\bibitem{N98} G. Navarro, {\it Characters and blocks of finite groups}, London Mathematical Society Lecture Note Series, 250, Cambridge Univ. Press, Cambridge, 1998. 

\bibitem{S08} J. Sangroniz, Restrictions of Brauer characters and $\pi$-partial characters, in {\it Ischia group theory 2008}, 236--242, World Sci. Publ., Hackensack, NJ.

\bibitem{S73} W.~A. Simpson\ and\ J.~S. Frame, The character tables for ${\rm SL}(3,\,q)$, ${\rm SU}(3,\,q^2)$, ${\rm PSL}(3,\,q)$, ${\rm PSU}(3,\,q^2)$, Canadian J. Math. {\bf 25} (1973), 486--494. 

\bibitem{TV20} H.~P. Tong-Viet, Conjugacy classes of $p$-elements and normal $p$-complements, Pacific J. Math. {\bf 308} (2020), no.~1, 207--222. 

\end{thebibliography}
\end{document}